\documentclass[12pt]{amsart}
\usepackage{amsmath}
\usepackage{amsthm}
\usepackage{amssymb}
\usepackage{amsfonts,mathrsfs}
\usepackage{stmaryrd}
\usepackage{amsxtra}  
\usepackage{epsfig}
\usepackage{verbatim}

\theoremstyle{plain}
\newtheorem{theorem}[equation]{Theorem}
\newtheorem{proposition}[equation]{Proposition}
\newtheorem{lemma}[equation]{Lemma}
\newtheorem{corollary}[equation]{Corollary}
\newtheorem{definition}[equation]{Definition}
\theoremstyle{remark}

\numberwithin{equation}{section}

\newcommand{\sjump}{\hskip .2 cm}

\newcommand{\dbarstar}{\bar{\partial}^{\star}}

\newcommand{\psum}{\sideset{}{^{\prime}}{\sum}}

\newcommand{\cn}{{\mathcal N}}
\newcommand{\dbar}{\bar \partial}

\begin{document}

\title{A smoothing property of the Bergman projection}
\author{A.-K. Herbig \& J. D. McNeal}
\thanks{Research of the first author was supported by FWF grants P19147 and AY0037721}
\thanks{Research  of the second author was partially supported by an NSF grant}
\address{Department of Mathematics, \newline University of Vienna, Vienna, Austria}
\email{anne-katrin.herbig@univie.ac.at}
\address{Department of Mathematics, \newline Ohio State University, Columbus, Ohio, USA}
\email{mcneal@math.ohio-state.edu}
\date{}
\begin{abstract} 
 Let $B$ be the Bergman projection associated to a domain $\Omega$  on which the $\bar\partial$-Neumann operator is compact.
We show that arbitrary $L^2$ derivatives of $Bf$ are controlled by derivatives of  $f$ taken in a single, distinguished direction. As a consequence, functions not contained in $C^{\infty}(\overline{\Omega})$ that are mapped by  $B$ to $C^{\infty}(\overline{\Omega})$ are explicitly described.
\end{abstract}
\maketitle

\section{Introduction}

The purpose of this paper is to prove a refined global regularity property of the Bergman operator associated to a pseudoconvex domain in $\mathbb{C}^{n}$ on which the $\dbar$-Neumann operator is compact.

Let $\Omega\subset\mathbb{C}^{n}$ be an open connected set with smooth boundary, $b\Omega$. 
If $\mathcal{O}(\Omega)$ and $L^{2}(\Omega)$ denote the spaces of holomorphic functions and Lebesgue square-integrable classes on $\Omega$, elementary estimates show that the Bergman space
$A^{2}(\Omega) =\mathcal{O}(\Omega)\cap L^{2}(\Omega)$ is closed in $L^{2}(\Omega)$. Consequently,  the orthogonal projection $B: L^{2}(\Omega)\longrightarrow A^{2}(\Omega)$ is a 
well-defined bounded operator on $L^{2}(\Omega)$. The operator $B$ is the Bergman projection. 

A problem of significant interest is to understand the behavior of $B$ on function spaces other than
$L^{2}(\Omega)$. Since $B$ generally lacks a simple closed-form expression as an integral operator, there cannot be a universal resolution to this problem. However, special cases have been widely studied and there are numerous results connected to this problem, for different classes of domains $\Omega$ and for different function spaces $F$; see \cite{Fefferman74, PhoSte77, AheSch79, McNeal89, NagRosSteWai89, McNeal91, McNeal94, McNSte94, Koenig02}  and their references for the principal cases known. The function spaces considered in these works are naturally graduated, e.g., Lipschitz spaces, Sobolev spaces, and the conclusion is that $B$ is a ``$0$th order operator'' on the graded family of Banach spaces considered, i.e., $B$ maps $F^{s}$ to $F^{s}$, preserving but not improving the scale $s$ within the family $\{F^{t}\}$. This is a sharp isotropic conclusion since $Bh=h$ for $h\in A^{2}(\Omega)$ and $A^{2}(\Omega)$ contains elements sharply in the classical spaces $F^{s}$.

There is however a basic non-isotropy inherent to the analysis of $\mathcal{O}(\Omega)$ which can lead to sharper boundedness results. This non-isotropy
arises from the partial complex  structure of $T(b\Omega)$, the tangent space to $b\Omega$, and goes back to the beginning of complex analysis in several variables (e.g., pseudoconvexity, Hartogs phenomena). The observation that this non-isotropy can lead to sharper boundedness results for $B$ and related operators, however, originates in the work of Stein \cite{Stein70,Stein72}. Subsequently, this fundamental observation has been developed, extended, and applied to a variety of problems in several complex variables by many mathematicians; see \cite{FolSte74,RotSte76} and the references to chapters  XII and XIII in \cite{Stein93}.

In this paper, we consider a simple, global aspect of the analysis of $B$ that seems to have been overlooked. Suppose that $\Omega\subset\mathbb{C}^{n}$ is a smoothly bounded domain on which the $\dbar$-Neumann operator, $N$, is compact. Under this hypothesis, Kohn and Nirenberg (Theorem 2 in \cite{KohNir65}) showed that $N$ is globally regular and that $B$ is globally regular as well, i.e., $B$ maps $C^{\infty}(\overline{\Omega})$ to itself.
More quantitatively, Kohn showed in \S 2 of \cite{Kohn84} that for any $k\in\mathbb{N}$ there exists a constant $C_{k}>0$ so that
\begin{align}\label{E:KohnBestimate}
  \|Bf\|_{k}\leq C_{k}\|f\|_{k}\qquad\sjump\forall\sjump f\in H^{k}(\Omega),
\end{align}
where $\|.\|_{k}$ denotes that $L^{2}$ Sobolev norm of order $k$ on $\Omega$. The goal of our paper is to improve \eqref{E:KohnBestimate} by showing that the full $H^{k}$ norm on the right hand side can be replaced by a term involving derivatives of $f$ in only one special direction (up to order $k$). Specifically, suppose $\Omega=\{z\in\mathbb{C}^{n}:r(z)<0\}$ for a $C^{\infty}$ function $r$ such that $\nabla r\neq 0$ on $\{z\in\mathbb{C}^{n}:r(z)=0\}$ and define the vector field
\begin{align}\label{D:badtangential}
X=\sum_{j=1}^{n}r_{\bar{z}_{j}}\frac{\partial}{\partial z_{j}}-r_{z_{j}}\frac{\partial}{\partial\bar{z}_{j}}.
\end{align}
Let $X^{\ell}f=X(X(\dots(Xf)\dots)$ denote $\ell$-fold differentiation of $f$ by the vector field $X$.
The main result of this paper is 
\begin{theorem}\label{T:MainTheorem}
  Let $\Omega\subset\mathbb{C}^{n}$ be a smoothly bounded, pseudoconvex domain on which the $\dbar$-Neumann operator is compact, $r$ a smooth defining function for $\Omega$, and $X$ the vector field in
\eqref{D:badtangential}.

Then for any $k\in\mathbb{N}$ there exists a constant $C_{k}>0$ such that
\begin{align}\label{E:Bestimate}
  \|Bf\|_{k}\leq C_{k}\sum_{\ell=0}^{k}\|X^{\ell}f\|
\end{align}
holds for all $f\in L^{2}(\Omega)$ such that $X^{\ell}f$ (in the distributional sense) belongs to $L^{2}(\Omega)$, $\ell\in\{1,\dots,k\}$.
\end{theorem}

The improvement over \eqref{E:KohnBestimate} given by Theorem \ref{T:MainTheorem} can be illustrated already on the unit disc $D$ in $\mathbb{C}$, $D=\{z\in\mathbb{C}: |z|<1\}$. Let
\begin{align*}
  \tilde{f}(z)=(1-|z|^{2})^{-\frac{1}{4}}\cdot g(z)
\end{align*}
for a fixed $g\in C^{\infty}(\overline{D})$, which does not vanish identically on $bD$. Then $\tilde{f}\in L^{2}(D)$ and 
$X\tilde{f}=(1-|z|^{2})^{-\frac{1}{4}}\cdot Xg$, since $X$ is tangent to $bD$, is also in $L^{2}(D)$. However, $\tilde{f}\notin H^{1}(D)$ since
\begin{align*}
  \frac{\partial}{\partial r}\tilde{f} \sim(1-|z|^{2})^{-\frac{5}{4}}\cdot g\notin L^{2}(D),
\end{align*}
where $\frac{\partial}{\partial r}$ denotes the radial derivative on $D$. Theorem \ref{T:MainTheorem} implies that $B\tilde{f}$ belongs to $H^{1}(D)$, but this cannot be concluded from
\eqref{E:KohnBestimate}.

\medskip

The proof of Theorem \ref{T:MainTheorem} will occur in two basic steps. Both steps involve the non-isotropy of $\mathcal{O}(\Omega)$ mentioned above, but in different ways. The first step is to show
\begin{align}\label{E:reductionholomorph}
   \|h\|_{k}\leq C_{k}\sum_{\ell=0}^{k}\|X^{\ell} h\|,\qquad h\in\mathcal{O}(\Omega).
\end{align}
This is a well-known fact, see, e.g., \cite{Barrett86, Boas87}, 
that follows from the ellipticity of the Cauchy-Riemann equations and does not require hypotheses on $\Omega$ other than the smoothness of $b\Omega$ (which can also be weakened, \cite{Barrett86}). Inequality \eqref{E:reductionholomorph} reduces establishing \eqref{E:Bestimate} to bounding $\sum_{\ell=0}^{k}\|X^{\ell}Bf\|$ by the right hand side of \eqref{E:Bestimate}. We emphasize that no properties of the Bergman projection beyond the fact that $Bf\in\mathcal{O}(\Omega)$ are used in this step.

The second step is to see how $X$ commutes with $B$.\footnote{If $\Omega$ is a ball in $\mathbb{C}^{n}$, this step simplifies enormously; see the Appendix.} In this step, the orthogonality of $B$ and the hypothesis that the $\dbar$-Neumann operator $N$ is compact are used heavily. Commuting $X$ directly with the integral operator $B$ is intractable because of the lack of useful estimates on the Bergman kernel under the hypothesis of Theorem \ref{T:MainTheorem}. Instead, Kohn's formula (see \eqref{E:Kohnsformula}) is used to transfer the problem to estimating derivatives of a $(0,1)$-form $\varphi$ closely connected to $f$, $\varphi=N\dbar f$. This idea was also used in \cite{BoaStr91,Kohn99,HerMcN06}  for similar reasons. A key point is to use Kohn's formula to expand only one side of the inner product
$\|XBf\|^{2}=(XBf,XBf)$. The other essential point is to use Proposition \ref{P:basicLestimate} of Section \ref{S:bad} to re-express derivatives in terms of $X$; a careful integration by parts argument then shows
\begin{align*}
  \|XBf\|\leq C\left(\|T\varphi\|+\|f\|_{1,X} \right)
\end{align*}
where $T$ is closely related to $X$ (but preserves $\text{Dom}(\dbarstar)$) and $\|.\|_{1,X}$ is defined in Section \ref{S:Proof}.

The final argument is to show that $\|T\varphi\|$ is dominated by $\|f\|_{1,X}$. The crucial estimate used here is the compactness estimate \eqref{E:compact}, with Proposition \ref{P:basicLestimate} again used to re-express derivatives of $\varphi$ as $T$ is commuted past the $\dbar$ and $\dbarstar$ operators. The special relationship between $\varphi$ and $f$ (namely, $\dbar\varphi=0$ and $\dbarstar\varphi=f-Bf$) then allows the proof of Theorem \ref{T:MainTheorem} to conclude for the case $k=1$. A careful examination of the estimates leading to the case $k=1$ gives the result for higher derivatives.

The paper is organized as follows. In Section \ref{S:Preliminaries} notation is fixed and some elements of the $\bar\partial$-Neumann theory are recalled. In Section \ref{S:bad}, basic geometric apparatus associated to the boundary of a smoothly bounded domain is laid out and the needed inequalities involving derivatives in the ``bad tangential direction" to the boundary are derived. The proof of Theorem \ref{T:MainTheorem} is then given in Section \ref{S:Proof}. In Section \ref{S:Remarks}, the class of functions which is mapped by $B$ to $C^{\infty}(\overline{\Omega})$ is described and the dependence of \eqref{E:Bestimate} on the choice of the vector field $X$ is discussed. The section concludes with a description of the class of vector fields which may replace $X$ in Theorem \ref{T:MainTheorem}.

We would like to thank Harold Boas for several stimulating conversations at the Erwin Schr\"odinger Institute in the fall of 2009. This project started when he suggested that estimates obtained in \cite{HerMcN06} should lead to improvements similar to Theorem \ref{T:MainTheorem}, even for more general domains than we consider here.

\medskip

\section{Preliminaries}\label{S:Preliminaries}
Standard coordinates on $\mathbb{C}^{n}$ will be denoted $(z_1,\dots ,z_n)$, with $z_k=x_k+\sqrt{-1} y_k$ for $k=1,\dots ,n$. The usual Cauchy-Riemann vector fields are defined
$$\frac{\partial}{\partial z_k}=\frac 12\left(\frac{\partial}{\partial x_k}-\sqrt{-1}\frac{\partial}{\partial y_k}\right),\qquad
\frac{\partial}{\partial\bar z_k}=\frac 12\left(\frac{\partial}{\partial x_k}+\sqrt{-1}\frac{\partial}{\partial y_k}\right).$$
Differentiation of a smooth function will often be denoted with subscripts, e.g., $f_{z_j}=\frac{\partial f}{\partial z_j}$.
 
Domains $\Omega\subset\mathbb{C}^{n}$ with smooth boundary $b\Omega$ will be described using defining functions: a smooth, $\mathbb{R}$-valued function $r$ is a defining function for $\Omega$ if $\Omega=\{z\in\mathbb{C}^{n}:r(z)<0\}$ and $\nabla r\neq 0$ on $\{z\in\mathbb{C}^{n}: r(z)=0\}$. 
Note that $\sum_{j=1}^n|r_{z_j}|^2\neq 0$ on $b\Omega$ if $r$ defines $\Omega$.
If $\Omega\subset\mathbb{C}^n$ is smoothly bounded, and $r$ defines $\Omega$, then $\Omega$ is pseudoconvex if for all $z\in b\Omega$
\begin{align}\label{E:pseudoconvex}
 \sum_{j,k=1}^{n} r_{z_{j}\bar{z}_{k}}(z)\xi_{j}\bar{\xi}_{k}\geq 0
\qquad\sjump\forall\sjump \xi\in\mathbb{C}^{n}\;\text{with}\;\sum_{j=1}^{n} r_{z_{j}}(z)\xi_{j}=0
\end{align}
holds.

Throughout the paper, $\|.\|$ denotes the $L^2(\Omega)$ norm and the $L^2$ Sobolev norm of order $k$, $k\in\mathbb{N}$, is denoted $\|.\|_k$. On functions these norms have a standard meaning, but on forms some ambiguity is involved. If $f$ and $g$ are measurable functions on $\Omega$, their $L^2$ inner product is
$$(f,g)=\int_\Omega f\,\bar g\, dV_E$$
where $dV_{E}$ denotes the Euclidean volume form (which we usually drop). The $L^2$ norm is then $\|f\|=\sqrt{(f,f)}$ and $L^2(\Omega)$ denotes the set of such $f$ with $\|f\|<\infty$. Also, $\|f\|_k=\sum_{|\alpha|=k}\left\|D^\alpha f\right\|$ where $D^\alpha$ is partial differentiation of order determined by the multi-index $\alpha$, taken in the sense of distributions, and $H^k(\Omega)$ denotes all $f$ such that $\|f\|_k <\infty$. On $(0,q)$-forms, we declare that elements of the basis given by $d\bar z_{i_1}\wedge\dots\wedge d\bar z_{i_q}$, $i_1 <\dots <i_q$, are (pointwise) orthonormal and define $\|.\|$ and $\|.\|_k$ as above on the components of a form relative to this basis. Thus, if a $(0,q)$-form $u$ is expressed as
\begin{align*}
  u(z)=\psum_{|J|=q}u_{J}(z)d\bar{z}^{J},\qquad z\in U,
\end{align*}
where $J$ is multi-index of length $q$, the $u_{J}$'s are functions, and $\sum_{|J|=q}'$  means that the sum is taken only over strictly increasing multi-indices, we set
$\|u\|=\psum_{|J|=q}\left\|u_{J}\right\|$ and $\|u\|_k=\psum_{|J|=q}\left\|u_{J}\right\|_k$.

\medskip

From the theory of the $\bar\partial$-Neumann problem, we shall recall only a few rudimentary facts needed in the proof of Theorem \ref{T:MainTheorem}. For proofs of these facts, and further information about the $\bar\partial$-Neumann problem, we refer the reader to
\cite{FolKoh72,Straube10}. 

The Cauchy-Riemann operator $\dbar: \Lambda^{0,q}(\overline{\Omega})\longrightarrow\Lambda^{0,q+1}(\overline{\Omega})$, defined initially on $(0,q)$-forms whose components belong to $C^{\infty}(\overline{\Omega})$, can be extended to an operator acting on $(0,q)$-forms with square-integrable components, $L_{0,q}^{2}(\Omega)$. An adjoint operator, $\bar\partial^*$, to $\dbar$ may then be defined using the standard $L_{0,q}^{2}(\Omega)$ inner product. The symbol $\vartheta$ will be used to denote the action of $\bar\partial^*$ but as a distributional differential operator; in particular, $\vartheta$ can be applied to any element of $L_{0,q}^{2}(\Omega)$.
The operator $\dbarstar$ comes with a naturally associated domain, $\text{Dom}(\dbarstar)$, so that, e.g., $u(z)=\sum_{j=1}^{n}u_{j}(z)\;d\bar{z}_{j}\in\Lambda^{0,1}(\overline{\Omega})$ belongs to $\text{Dom}(\dbarstar)$ if
\begin{align}\label{E:domaindbarstar}
  \sum_{j=1}^{n}r_{z_{j}}(z)u_{j}(z)=0\qquad\sjump\forall\sjump z\in b\Omega
\end{align}
holds.  The space $\Lambda^{0,1}(\overline{\Omega})\cap\text{Dom}(\dbarstar)$ is denoted by $\mathcal{D}^{0,1}(\Omega)$.

\medskip

The complex Laplacian is defined as $\Box u=\dbar\dbarstar u+\dbarstar\dbar u$ for those $u\in \text{Dom}(\dbar)\cap \text{Dom}(\dbarstar)$ for which $\dbar u\in \text{Dom}(\dbarstar)$ and $\dbarstar u\in \text{Dom}(\dbar)$ holds; the set of these forms is denoted by $\text{Dom}(\Box)$. 
For the purpose at hand, only the complex Laplacian on $(0,1)$-forms is of relevance, so that we restrict our considerations to this bi-degree.

The Dirichlet form associated to the $(\Box,\text{Dom}(\Box))$ boundary value problem,
$$Q(u,u)=\|\dbar u\|^{2}+\|\dbarstar u\|^{2},$$
satisfies the following basic estimate, with a constant independent of $u$, if $\Omega$ is pseudoconvex:
\begin{align}\label{E:basicestimate}
  \|u\|^{2}+\sum_{k=1}^{n}\left\|\frac{\partial u}{\partial\bar{z}_{k}}\right\|^{2}\lesssim Q(u,u)\qquad\sjump\forall\sjump
  u\in\mathcal{D}^{0,1}(\Omega)
\end{align}
where the expression $\frac{\partial u}{\partial \bar{z}_{k}}$ denotes $\frac{\partial}{\partial\bar{z}_{k}}$ acting on each component of $u$.
Inequality \eqref{E:basicestimate} implies that the $\dbar$-Neumann operator $N: L_{0,1}^{2}(\Omega)\longrightarrow \text{Dom}(\Box)$, satisfying $N\circ\Box=Id_{|\text{Dom}(\Box)}$ and $\Box\circ N=Id$, exists and is bounded on $L_{0,1}^{2}(\Omega)$. If $N$ satisfies the stronger condition that it is a {\it compact} operator, Kohn and Nirenberg \cite{KohNir65} showed that $N$ preserves the class of $(0,1)$-forms that are smooth up to the boundary. They also showed that compactness of $N$ is equivalent to the following collection of inequalities: for every $\epsilon >0$, there exists a constant $C(\epsilon)$ such that

\begin{equation}\label{E:compact}
\|u\|^2\leq \epsilon\,\, Q(u,u)+ C(\epsilon)|u|_{-1},\qquad\forall\sjump u\in \mathcal{D}^{0,1}(\Omega)
\end{equation}
holds, where $|.|_{-1}$ denotes the ordinary Sobolev norm of order -1 (or any norm compact with respect to $\|.\|$). The property that $N$ is compact does not hold on all smoothly bounded pseudoconvex domains, but many classes of such domains are known to have this property, see \cite{FuStr01}.

The relationship between $N$ and the Bergman projection $B$ which we will use is expressed by Kohn's formula:
\begin{align}\label{E:Kohnsformula}
  B=Id-\dbarstar N\dbar.
\end{align}

We refer to the inequality $2|ab|\leq \epsilon a^{2}+\frac{1}{\epsilon}b^{2}$, $a,b\in\mathbb{R}$ and $\epsilon>0$, as the (sc)-(lc) inequality. Also, the expression $|a|\lesssim|b|$ will mean that there exists a constant $K>0$, independent of certain parameters (that will be specified or clear when used), such that
$|a|\leq K |b|$ holds. The usual notation for the commutator of two operators $A$ and $B$ is used: $[A,B]=AB-BA$.

\section{Analysis in the bad direction}\label{S:bad}

The Cauchy-Riemann fields induce a splitting of the complexified tangent bundle to $\mathbb{C}^{n}$, $\mathbb{C}T\left(\mathbb{C}^{n}\right)=T\left(\mathbb{C}^{n}\right)\otimes\mathbb{C}$, written $$\mathbb{C}T\left(\mathbb{C}^{n}\right)=T^{1,0}\left(\mathbb{C}^{n}\right)\oplus T^{0,1}\left(\mathbb{C}^{n}\right)$$
where $T^{1,0}\left(\mathbb{C}^{n}\right)$ denotes the span of $\bigl\{\frac{\partial}{\partial z_1},\dots, \frac{\partial}{\partial z_n}\bigr\}$ and 
$T^{0,1}\left(\mathbb{C}^{n}\right)=\text{span}\bigl\{\frac{\partial}{\partial\bar z_1},\dots, \frac{\partial}{\partial\bar z_n}\bigr\}$.
The real tangent bundle to $b\Omega$, $T\left(b\Omega\right)$, also splits via the Cauchy-Riemann structure. Let $T^{1,0}\left(b\Omega\right)=T\left(b\Omega\right)\cap T^{1,0}\left(\mathbb{C}^{n}\right)$ and $T^{0,1}\left(b\Omega\right)=T\left(b\Omega\right)\cap T^{0,1}\left(\mathbb{C}^{n}\right)$. Then 
$$T\left(b\Omega\right)=T^{1,0}\left(b\Omega\right)\oplus T^{0,1}\left(b\Omega\right)\oplus\mathcal{B}$$
where $\mathcal{B}$, satisfying $\text{dim}_{\mathbb{R}}\mathcal{B}=1$, is the ``bad direction'' tangent to $b\Omega$. 
Note that $\mathcal{B}$ is spanned by the vector field $X$ defined in \eqref{D:badtangential}.

The globally defined vector fields
$$\cn=\sum_{j=1}^n r_{\bar z_j}\frac{\partial}{\partial z_j}\quad\text{  and  }\quad\overline{\cn}=\sum_{j=1}^n r_{z_j}\frac{\partial}{\partial\bar z_j},$$
in $T^{1,0}\left(\mathbb{C}^{n}\right)$ and $T^{0,1}\left(\mathbb{C}^{n}\right)$ respectively, are transverse to $b\Omega$ since 
$\cn (r)=\overline{\cn} (r)=\sum_{j=1}^n |r_{z_j}|^2\neq 0$ on $b\Omega$. 
For $k\in\{1,\dots, n\}$, define the vector fields

\begin{equation}\label{E:L_k}
L_k=\frac{\partial}{\partial z_k}-\chi\cdot r_{z_k}\cn
\end{equation}
and their conjugates $\overline{L}_k$, where $\chi\in C^{\infty}(\overline{\Omega})$ equals $(\sum_{j=1}^{n}|r_{z_{j}}|^{2})^{-1}$ in some neighborhood of $b\Omega$.

\begin{proposition}\label{P:tangentspan}
$$T^{1,0}\left(b\Omega\right)=\text{\rm span}\left\{L_1,\dots, L_n\right\}.$$
Similarly, $T^{0,1}\left(b\Omega\right)$ is spanned by $\left\{\overline{L}_k: k=1,\dots, n\right\}$.
\end{proposition}

\begin{proof} Note that $L_k(r)=0$ for $k=1,\dots, n$, so each $L_k$ is tangent to $b\Omega$. Since \eqref{E:L_k} implies that 
$\frac{\partial}{\partial z_k}=L_k+\chi\cdot r_{z_k}\cn$ for $k=1,\dots, n$, the statement $T^{1,0}\left(\mathbb{C}^{n}\right)=\text{\rm span}\left\{L_1,\dots, L_n, \cn\right\}$ is immediate. The fact that $\cn(r)\neq 0$ on $b\Omega$ then yields the claimed statement about $T^{1,0}\left(b\Omega\right)$. The statement about $T^{0,1}\left(b\Omega\right)$ is proved analogously.
\end{proof}

A useful frame for all $\mathbb{C}T\left(\mathbb{C}^{n}\right)$, that contains the fields \eqref{E:L_k}, may also be isolated.

\begin{proposition}\label{P:fullspan}

$$\mathbb{C}T\left(\mathbb{C}^{n}\right)=\text{\rm span}\left\{L_1,\dots, L_n, \overline{L}_1,\dots, \overline{L}_n,X, \overline{\cn}\right\}.$$
\end{proposition}

\begin{proof} The proof is similar to Proposition 2.2.
\end{proof}

\medskip
\noindent{\it Remark:} Note that Propositions 2.2 and 2.3 give more spanning vectors than the dimensions of the spaces $T^{1,0}\left(b\Omega\right)$ and 
$\mathbb{C}T\left(\mathbb{C}^{n}\right)$ require. This over-prescription of spanning vectors will simplify some technicalities in Section \ref{S:Proof}.
\medskip

The following integration by parts result will play a crucial role in Section \ref{S:Proof}.

\begin{proposition}\label{P:basicLestimate} Let $\Omega\subset\mathbb{C}^n$ be a smoothly bounded domain. 
 Then  for all $g\in C^{\infty}(\overline{\Omega})$ and $\epsilon>0$
  \begin{align}\label{E:basicLestimate}
    \sum_{j=1}^{n}\|L_{j}g\|^{2}\lesssim\epsilon\|Xg\|^{2}+\frac{1}{\epsilon}\|g\|^{2}+\sum_{j=1}^{n}\|\overline{L}_{j}g\|^{2}
    +\epsilon\|r\overline{\mathcal{N}}g\|^{2}
  \end{align}
  holds with a constant independent of $g$.
\end{proposition}

\begin{proof}
 Let $L_{j}^{*}$ be the formal adjoint of $L_{j}$, $j\in\{1,\dots,n\}$. Then,
  since the $L_{j}$'s are tangential to $b\Omega$, integration by parts yields
  \begin{align*}
    \sum_{j=1}^{n}\|L_{j}g\|^{2}&=\sum_{j=1}\left(L_{j}^{*}L_{j}g,g\right)\\
    &\lesssim\sum_{j=1}^{n}\left\{\left|(\overline{L}_{j}L_{j}g,g) \right|+\|L_{j}g\|\cdot\|g\|
    \right\},
  \end{align*}
  where the last estimate follows since $L_j^*$ equals $-\bar L_j $ modulo multiplication by a function in $C^\infty(\overline\Omega)$.
  Applying the (sc)-(lc) inequality then gives
  \begin{align*}
    \sum_{j=1}^{n}\|L_{j}g\|^{2}&\lesssim \sum_{j=1}^{n}\left|\left(\overline{L}_{j}L_{j}g,g \right) \right|\\
    &\lesssim \sum_{j=1}^{n}\left\{
      \left|\left([\overline{L}_{j},L_{j}]g,g \right) \right|+\left|\left(\overline{L}_{j}g,L^{*}_{j}g \right) \right|
    \right\}\\
    &\lesssim\sum_{j=1}^{n}\left\{
     \left|\left([\overline{L}_{j},L_{j}]g,g \right) \right|+\left\|\overline{L}_{j}g\right\|\cdot\|g\|+\left\|\overline{L}_{j}g\right\|^{2}
    \right\}.
  \end{align*}
  Since $[\overline{L}_{j},L_{j}]\in T(b\Omega)$, Propositions \ref{P:tangentspan} and \ref{P:fullspan} imply that
   $[\overline{L}_{j},L_{j}]=\sum a_kL_k+\sum b_k\overline{L}_k +c_{1}X+c_{2}r\overline{\mathcal{N}}$. 
  Therefore,
  \begin{align*}
    \sum_{j=1}^{n}\|L_{j}g\|^{2}\lesssim
    \left(\|Xg\|+\|r\overline{\mathcal{N}}g\|+\sum_{j=1}^{n}\left\{\|L_{j}g\|+\|\overline{L}_{j}g\| \right\}\right)\cdot\|g\|
    +\sum_{j=1}^{n}\|\overline{L}_{j}g\|^{2}.
  \end{align*}
  Using the (sc)-(lc) inequality, we obtain for any $\epsilon>0$ that
  \begin{align*}
    \sum_{j=1}^{n}\|L_{j}g\|^{2}\lesssim\epsilon\|Xg\|^{2}+\frac{1}{\epsilon}\|g\|^{2}+\sum_{j=1}^{n}\|\overline{L}_{j}g\|^{2}+
    \epsilon\|r\overline{\mathcal{N}}g\|^{2}.
  \end{align*}
  \end{proof}
  
  \begin{corollary}\label{C:basicSobolev1estimate} 
  Let $\Omega\subset\mathbb{C}^{n}$ be a smoothly bounded domain. Then:
  \begin{itemize}
   \item[(i)] For all $g\in\mathcal{C}^{\infty}(\overline{\Omega})$
  $$ \|g\|_{1}^{2}\lesssim\|Xg\|^{2}+\|g\|^{2}+\sum_{j=1}^{n}\|\overline{L}_{j}g\|^{2}+\|\overline{\cn}g\|^{2}$$
  holds, where the constant is independent of $g$.
  \item[(ii)] For all $h\in\mathcal{O}(\Omega)\cap C^{\infty}(\overline{\Omega})$
    $$ \|h\|_{k}^{2}\lesssim\sum_{\ell=0}^{k}\|X^{\ell}h\|^{2}$$
  holds with a constant independent of $h$.
  \end{itemize}
  \end{corollary}
  
  \begin{proof}
    By Proposition \ref{P:fullspan} any first order derivative can be written as a linear combination of $\overline{\cn}$, $X$, the $L_{j}$'s and
    the $\overline{L}_{j}$'s. An application of Proposition \ref{P:basicLestimate} then completes the proof of (i).
    
    The case $k=1$ of (ii) follows immediately from (i) and $h\in\mathcal{O}(\Omega)$.
    Write $D^{\alpha}$ for $\frac{\partial^{|\alpha|}}{\partial z^{\alpha}}$ for $\alpha=(\alpha_{1},\dots,\alpha_{n})$ a multi-index.
     Then $h\in\mathcal{O}(\Omega)$ also implies that for given integer $k>1$
    \begin{align*}
      \|h\|_{k}^{2}\lesssim\sum_{|\alpha|\leq k}\left\|D^{\alpha}h \right\|^{2}\lesssim
      \sum_{|\alpha|=k}\left\|D^{\alpha}h \right\|^{2}+\|h\|_{k-1}^{2}
    \end{align*}
    holds. For given $\alpha$ with $|\alpha|=k$ 
    choose $i$ such that $\alpha_{i}>0$ and set
    $\alpha'=(\alpha_{1},\dots,\alpha_{i-1},\alpha_{i}-1,\alpha_{i+1},\dots,\alpha_{n})$. Then, since $\frac{\partial}{\partial z_{i}}=
    L_{i}+\chi\cdot r_{z_{i}}\overline{\cn}$, it follows that
    \begin{align*}
      \left\|D^{\alpha}h\right\|^{2}
      =\bigl\|D^{\alpha'}L_{i}h \bigr\|^{2}\lesssim\bigl\|L_{i}D^{\alpha'}h \bigr\|^{2}
      +\|h\|_{k-1}^{2}.
    \end{align*}
    Part (i) and $h\in\mathcal{O}(\Omega)$ yield
    \begin{align}\label{E:alphareduction}
      \bigl\|L_{i}D^{\alpha'}h \bigr\|^{2}&\lesssim\bigl\|XD^{\alpha'}h \bigr\|^{2}+\bigl\|D^{\alpha'}h \bigr\|^{2}
    +\sum_{j=1}^{n}\bigl\|\overline{L}_{i}D^{\alpha'}h \bigr\|^{2}+\bigl\|\overline{\cn}D^{\alpha'}h \bigr\|^{2}\notag\\
    &\lesssim\bigl\|XD^{\alpha'}h \bigr\|^{2}+\|h\|_{k-1}^{2}.
    \end{align}
    Repeating the arguments leading up to \eqref{E:alphareduction} $\alpha_{j}$-times for all $\alpha_{j}$ gives
    $$\|h\|_{k}^{2}\lesssim\sum_{\ell=0}^{k}\|X^{\ell}h\|^{2}+\|h\|_{k-1}^{2}.$$
    The claimed inequality (ii) now follows by induction on $k$.
  \end{proof}

\medskip

\section{Proof of Theorem \ref{T:MainTheorem}}\label{S:Proof}

The following spaces of Sobolev type occur in the conclusion of Theorem \ref{T:MainTheorem}:

\begin{definition}\label{D:H_X}
Let $\Omega\subset\mathbb{C}^{n}$ be smoothly bounded, $r$ a defining function for $\Omega$, and $X$ the vector field given by \eqref{D:badtangential}.

Define, for $k\in\mathbb{N}$,
\begin{align*}
H_{X}^{k}(\Omega)=\left\{f\in L^{2}(\Omega): \,\, X^{j}f \text{(as a distribution)} \in L^{2}(\Omega), j\in\{1,\dots,k\}\right\}.
\end{align*}
\end{definition}

For fixed $k$, $H_{X}^{k}(\Omega)$ is a Hilbert space with respect to the inner product
\begin{align*}
  (f,g)_{k,X}:=\sum_{j=0}^{k}\left(X^{j}f,X^{j}g\right)\qquad\sjump\forall\sjump f,g\in H_{X}^{k}(\Omega),
\end{align*}
and $C^{\infty}(\overline{\Omega})$ is dense in $H_{X}^{k}(\Omega)$ with respect to the norm $\|.\|_{k,X}$ induced by this inner product.
The completeness and density assertions are proved in the same way as for ordinary Sobolev spaces, see e.g.,  in \cite[Theorem 2 in \S 5.2, Theorem 3 in \S 5.3.3]{Evans98}.

\begin{lemma}\label{L:apriori} Suppose the hypotheses of Theorem \ref{T:MainTheorem}. If for each $k\in\mathbb{N}$ the inequality
\begin{equation}\label{E:apriori}
\|Bf\|_k\lesssim \|f\|_{k,X},\qquad\forall\sjump Bf, f\in C^\infty(\overline{\Omega}),
\end{equation}
holds, where the constant in $\lesssim$ does not depend on $f$, then the conclusion of Theorem \ref{T:MainTheorem} holds.
\end{lemma}

\begin{proof}
Suppose $g\in H_X^k(\Omega)$. Since $C^{\infty}(\overline{\Omega})$ is dense in $H_{X}^{k}(\Omega)$, there exists $\left\{f_j\right\}\subset C^{\infty}(\overline{\Omega})$ such that $\lim_{j\to\infty}\|f_j- g\|_{k,X} =0$; in particular, the sequence $\left\{f_j\right\}$ is Cauchy with respect to the norm 
$\|.\|_{k,X}$.  The proposition in \S2 in \cite{Kohn84} says that $Bf_j\in C^{\infty}(\overline{\Omega})$, so \eqref{E:apriori} is applicable.
Since \eqref{E:apriori} holds with constant independent of $f$,

$$\|Bf_l -Bf_m\|_k= \left\| B\left(f_l-f_m\right)\right\|\lesssim\|f_l-f_m\|_{k,X}$$
holds with constant independent of $l$ and $m$, which
implies that $\left\{Bf_j\right\}$ is a Cauchy sequence in $H^k(\Omega)$. Hence, there exists $F\in H^k(\Omega)$ such that $\lim_{j\to\infty}\|Bf_j- F\|_{k} =0$. We claim that $F=Bg$ in $H^{k}(\Omega)$. First note that
\begin{align*}
  \|F-Bg\|&\leq \|F-Bf_{j}\|+\|Bf_{j}-Bg\|\\
  &\lesssim \|F-Bf_{j}\|+\|f_{j}-g\|
\end{align*} 
holds, and hence $F=Bg$ in $L^{2}(\Omega)$. Moreover, since
\begin{align*}
  \left|(F-Bg,\psi)\right|\leq \|F-Bg\|\cdot\|\psi\|\quad\sjump\forall\sjump\psi\in C^{\infty}_{c}(\Omega),
\end{align*}
it then follows that $(F-Bg,\psi)=0$ for all $\psi\in C^{\infty}_{c}(\Omega)$.  Thus, all the derivatives of $Bg$ up to order $k$ exist in the distributional sense, and in fact equal those of $F$. Since $F\in H^{k}(\Omega)$, this implies that $F=Bg$ in $H^{k}(\Omega)$. 

\end{proof}

The collection of estimates \eqref{E:apriori} will be shown by induction on $k$. We start with the proof of
  \begin{align}\label{E:apriori1}
    \|Bf\|_{1}\lesssim \|f\|_{1,X}\qquad\sjump\forall\sjump f\in C^{\infty}(\overline{\Omega}),
  \end{align}  
i.e., the case $k=1$ in \eqref{E:apriori}, and break establishing \eqref{E:apriori1} into a few lemmas in order to separate the different elements of the proof.

\begin{lemma} For all $f\in C^{\infty}(\overline{\Omega})$, the inequality
\begin{align}\label{E:XBfreduction}
  \left\|Bf\right\|_1\lesssim \|XBf\| + \|f\|
\end{align}  
holds with a constant independent of $f$.
\end{lemma}

\begin{proof}
Since $Bf\in\mathcal{O}(\Omega)\cap C^{\infty}(\overline{\Omega})$, part (ii)  of Corollary \ref{C:basicSobolev1estimate} applies with $k=1$ so that 
 $$ \|Bf\|_{1}^{2}\lesssim\|XBf\|^{2}+\|Bf\|^{2}$$
 holds. The boundedness of $B$ in $L^{2}$ then yields \eqref{E:XBfreduction}.
  \end{proof}
  
To bring estimate \eqref{E:compact} to bear, set $\varphi =N\dbar f$ and note that $Bf=f-\dbarstar\varphi$ by Kohn's formula \eqref{E:Kohnsformula}. Also note that $\varphi\in \mathcal{D}^{0,1}(\Omega)$, since $N$ maps into $\text{Dom}(\dbar^*)$ and $N$ is globally regular.

\begin{lemma}\label{L:basicphiestimates} For $f\in C^{\infty}(\overline{\Omega})$ and $\varphi:=N\dbar f$, the inequalities

\begin{itemize}
  \item[(i)] $Q(\varphi,\varphi)\leq\|f\|^{2}$,
  \medskip
  \item[(ii)] $\|\varphi\|^{2}+\sum_{j=1}^{n}\|\overline{L}_{j}\varphi\|^{2}+\|\overline{\cn}\varphi\|^{2}\lesssim\|f\|^{2}$,
  \medskip
  \item[(iii)] $\|\varphi\|_{1}^{2}\lesssim\|X\varphi\|^{2}+\|f\|^{2}$
\end{itemize}
hold with constants independent of $f$.
\end{lemma}
\begin{proof}
  Part (i) follows from
  \begin{align*}
    Q(\varphi,\varphi)=\|\dbar\varphi\|^{2}+\|\dbarstar\varphi\|^{2}=\|f-Bf\|^{2}
  \end{align*}
  and the orthogonality of $B$.
  Since (i) holds, (ii) follows from \eqref{E:basicestimate}. Lastly, applying part (i) of Corollary  \ref{C:basicSobolev1estimate} to each component of $\varphi$ and using (ii) yields (iii).
\end{proof}

The vector field $X$ does not preserve membership in $\mathcal{D}^{0,1}(\Omega)$. Thus a slight modification of $X$ is needed before proceeding with the proof of \eqref{E:apriori1}. 

\begin{lemma}\label{L:T} There exists an operator $T$ such that
    \begin{itemize}
    \item[(i)] $T u\in\mathcal{D}^{0,1}(\Omega)$ 
    whenever $u\in\mathcal{D}^{0,1}(\Omega)$, and
    \item[(ii)] $\| (X-T)u\|_{i}\lesssim \|u\|_{i}$ for $u\in\Lambda^{0,1}(\overline{\Omega})$ for $i\in\{0,1\}$.
    \end{itemize}
    \end{lemma}
  
  \begin{proof}
  Since $X$ is tangential to $b\Omega$, it follows from \eqref{E:domaindbarstar} that
  \begin{align*}
   \sum_{j=1}^{n}r_{z_{j}}Xu_{j}=-\sum_{j=1}^{n}X(r_{z_{j}})u_{j}\qquad\sjump\forall\sjump u\in\mathcal{D}^{0,1}(\Omega)
  \end{align*}
  holds on $b\Omega$. Thus, defining
  \begin{align}\label{D:T}
    Tu:=Xu+\dbar r\cdot\chi\sum_{j=1}^{n}X(r_{z_{j}})u_{j}\qquad\sjump\forall\sjump u\in\Lambda^{0,1}(\overline{\Omega})
  \end{align}
  for $\chi$ defined as in \eqref{E:L_k},
  yields that $Tu\in\mathcal{D}^{0,1}(\Omega)$ whenever $u\in\mathcal{D}^{0,1}(\Omega)$. Moreover, on the first order level $T$ acts componentwise on $u$  and equals $X$. Since
\begin{align*}
 \|(X-T)u\|_{i}=\sum_{k=1}^{n}\Bigl\|r_{\bar{z}_{k}}\chi\sum_{j=1}^{n}X(r_{z_{j}})u_{j}\Bigr\|_{i}\lesssim\|u\|_{i}
\end{align*}
clearly holds for $i\in\{0,1\}$, the proof is complete. 
\end{proof}

The next result shows that only a single directional derivative of $\varphi$ is needed to control $XBf$.

\begin{lemma}\label{L:XBfbyTphi} The inequality
\begin{align}\label{E:XBfTphi}
   \|XBf\|^{2}&\lesssim \|T\varphi\|^{2}+ \|f\|_{1,X}^{2}
\end{align}
holds, with constant independent of $f\in C^{\infty}(\overline{\Omega})$.
\end{lemma}

\begin{proof}
Since
\begin{align*}
  \|XBf\|^{2}=(Xf,XBf)-(X\dbarstar\varphi,XBf),
\end{align*}
the (sc)-(lc) inequality gives
\begin{align}\label{E:XBfXphi}
  \|XBf\|^{2}&\lesssim \|Xf\|^{2}+\left|(X\dbarstar\varphi,XBf) \right|\notag\\
  &\leq\left|\left([X,\vartheta]\varphi , XBf\right) \right|+\left|\left(\vartheta X\varphi,XBf\right)  \right|+\|Xf\|^{2}\notag\\
  &\lesssim \|\varphi\|_{1}\cdot\|XBf\|+\left|\left(\vartheta X\varphi,XBf\right)  \right|+\|Xf\|^{2}.
\end{align}
Furthermore,
\begin{align*}
  \left|\left(\vartheta X\varphi,XBf\right)  \right|&\leq\left|\left(\dbarstar T\varphi,XBf \right) \right|
  +\left|\left(\vartheta\left((X-T)\varphi\right),XBf \right) \right|\\
  &\lesssim \left|(T\varphi,\dbar XBf \right|+\left\|(X-T)\varphi \right\|_{1}\cdot\|XBf\|\\
  &\lesssim \left|\left(T\varphi,[\dbar,X]Bf\right)\right|+\|\varphi\|_{1}\cdot\|XBf\|,
\end{align*}
where the last step follows from $Bf\in\mathcal{O}(\Omega)$ and (ii) of Lemma \ref{L:T}. Inequality \eqref{E:XBfreduction} then implies 
\begin{align*}
  \left|\left(\vartheta X\varphi,XBf\right)  \right|
  \lesssim
  \|\varphi\|_{1}\cdot\left(\|XBf\|+\|f\|\right).
\end{align*}
Substituting this into \eqref{E:XBfXphi} yields
\begin{align*}
   \|XBf\|^{2}&\lesssim \|\varphi\|_{1}\cdot\left(\|XBf\|+\|f\|\right)+ \|f\|_{1,X}^{2}.
\end{align*}
Since (iii) of Lemma \ref{L:basicphiestimates} together with (ii) of Lemma \ref{L:T} imply that
\begin{align*}
  \|\varphi\|_{1}^{2}\lesssim \|T\varphi\|^{2}+\|f\|^{2},
\end{align*}
the desired \eqref{E:XBfTphi} follows.
\end{proof}

To complete the proof of \eqref{E:apriori1}, it remains to establish the following

\begin{lemma} For $f\in C^{\infty}(\overline{\Omega})$ and $\varphi=N\dbar f$, the inequality
\begin{equation}\label{E:Tphibyf_X}
\|T\varphi\|^{2}\lesssim \|f\|_{1,X}^{2}
\end{equation}
holds with a constant independent of $f$.
\end{lemma} 

\begin{proof} For any $\epsilon >0$, \eqref{E:compact} implies

\begin{align}\label{E:startcompact}
\left\| T\varphi\right\|^{2}&\leq\epsilon\,Q\left(T\varphi,T\varphi\right)+ C(\epsilon)\left|T\varphi\right|_{-1}^2\notag \\
&\leq\epsilon\,Q\left(T\varphi,T\varphi\right)+\tilde C_1(\epsilon)\|\varphi\|^2\notag \\
&\leq\epsilon\left[\left\|\dbar T\varphi\right\|^{2}+ \left\|\dbarstar T\varphi\right\|^{2}\right]+\tilde C_2(\epsilon)\|f\|^2,
\end{align}
where Lemma \ref{L:basicphiestimates} (ii) is applied in the last line. The two terms in the bracket expression will be estimated separately.

The operator $T$ has only been defined on $(0,1)$-forms. Thus, to commute $\dbar$ (and $\dbarstar$) by the leading term of $T$, we switch back to $X$. That is,
  $$\|\dbar T\varphi\|^{2}\lesssim \|\dbar X\varphi\|^{2}+\|\dbar\left((X-T)\varphi\right)\|^{2}.$$
By part (ii) of Lemma \ref{L:T} it follows that
$$\|\dbar\left((X-T)\varphi\right)\|^{2}\lesssim\|\varphi\|_{1}^{2},$$
while $\dbar\varphi=0$ gives
$$\|\dbar X\varphi\|^{2}=\|[\dbar,X]\varphi\|^{2}\lesssim\|\varphi\|_{1}^{2}.$$
It then follows from (iii) of Lemma \ref{L:basicphiestimates} that
\begin{equation}\label{E:dbarcompact}
\left\|\dbar T\varphi\right\|^2\lesssim \left\|X\varphi\right\|^2 + \|f\|^2.
\end{equation}

\medskip
The estimates when commuting (the leading term of) $T$ past $\dbarstar$ are slightly different:

\begin{align*}
  \|\dbarstar T\varphi\|^{2}&\lesssim \|\vartheta X\varphi\|^{2}+\|\vartheta(X-T)\varphi\|^{2}\\
  &\lesssim\|\vartheta X\varphi\|^{2}+\|\varphi\|_{1}^{2},
\end{align*}
where the second estimate follows from part (ii) of Lemma \ref{L:T}.  Commuting $X$ by $\vartheta$ yields
$$\|\vartheta X\varphi\|^{2}\lesssim\|[\vartheta,X]\varphi\|^{2}+\|X\vartheta\varphi\|^{2}\lesssim\|\varphi\|_{1}^{2}+\|X\dbarstar\varphi\|^{2}$$
so that
\begin{align}\label{E:dbarstarcompact}
  \|\dbarstar T\varphi\|^{2}&\lesssim\|\varphi\|_{1}^{2}+\|X(f-Bf)\|^{2}\notag\\
  &\lesssim \|X\varphi\|^{2}+\|f\|^{2}+\|X(f-Bf)\|^{2}
\end{align}
by (iii) of Lemma \ref{L:basicphiestimates}.
For the last term in \eqref{E:dbarstarcompact},
 Lemma \ref{L:XBfbyTphi}  implies
 $$ \left\|Xf\right\|^2+\left\|XBf\right\|^2 
  \lesssim \left\|T\varphi\right\|^2+\|f\|^2_{1,X}.$$
\medskip

Combining \eqref{E:dbarcompact} and \eqref{E:dbarstarcompact} and using (ii) of Lemma \ref{L:T} to re-write the terms involving $X\varphi$, we have
$$ Q\left( T\varphi, T\varphi\right)\lesssim\left\|T\varphi\right\|^2+\|f\|^2_{1,X}.$$
Returning to \eqref{E:startcompact} and choosing $\epsilon$ small enough to absorb $\left\|T\varphi\right\|^2$ into the left hand side gives the claimed estimate \eqref{E:Tphibyf_X}.
\end{proof}

\medskip

\medskip

It remains to show that \eqref{E:apriori} holds for $k>1$. Let $k>1$ be a fixed integer, and suppose that
\begin{align}\label{E:inductionhypo}
  \left\|Bf\right\|_{\ell}\lesssim\left\|f\right\|_{\ell,X}\qquad\sjump\forall\sjump f\in C^{\infty}(\overline{\Omega})
\end{align}
holds for all $\ell\leq k-1$. The initial step of the proof that
\begin{align}\label{E:ktoprove}
   \left\|Bf\right\|_{k}\lesssim\left\|f\right\|_{k,X}\qquad\sjump\forall\sjump f\in C^{\infty}(\overline{\Omega})
\end{align}
holds is the same as when $k=1$.

\begin{lemma}\label{L:XkBfreduction}
  For all $f\in C^{\infty}(\overline{\Omega})$ the inequalities
  \begin{align*}
    \left\|Bf\right\|_{k}\lesssim\left\|Bf\right\|_{k,X}\lesssim\left\|X^{k}Bf\right\|+\|f\|_{k-1,X}
  \end{align*}
  hold with constants independent of $f$.
\end{lemma}

\begin{proof}
  The first inequality holds by part (ii) of Corollary \ref{C:basicSobolev1estimate} since $Bf\in\mathcal{O}(\Omega)\cap
  C^{\infty}(\overline{\Omega})$. The second follows from the induction hypothesis \eqref{E:inductionhypo}.
\end{proof}

The new issue for $k>1$ is to verify that only derivatives involving $X$ appear in the (essential) terms occurring when
the higher-order operator $X^{k}$ is passed by $B$. The following two lemmas will be used. 

\begin{lemma}\label{L:AB}
 Let $A$ and $B$ be differential operators of order $a$ and $b$, respectively. Set
 \begin{align*}
   C_{i}=\underbrace{\left[\dots\left[\left[A,B\right],B\right] ,\dots,B\right]}_{i\;\text{commutators}},
   \;\;\widetilde{C}_{i}=\underbrace{\left[B,\dots,\left[B,\left[A,B\right]\right]\dots\right]}_{i\;\text{commutators}}.
 \end{align*}
 Then both $C_{i}$ and $\widetilde{C}_{i}$ are of order $a+i(b-1)$. Moreover, for any $m\in\mathbb{N}$
 \begin{align*}
   \left[A,B^{m}\right]=\sum_{j=0}^{m-1}{m\choose j}B^{j}C_{m-j}=\sum_{j=0}^{m-1}{m\choose j}\widetilde{C}_{m-j}B^{j}.
 \end{align*}
\end{lemma}
\begin{proof}
The conclusion follows by induction on $m$ in a straightforward manner.
\end{proof}

\begin{lemma}\label{L:Tm}
 The vector field $X$ and the operator $T$ as defined in \eqref{D:T} satisfy
 \begin{align}\label{E:Tm}
  \left\|\left(X^{m}-T^{m}\right)u\right\|_{1}\lesssim\sum_{\ell=0}^{m-1}\left\|T^{\ell}u\right\|_{1}
 \end{align}
 as well as
 \begin{align}\label{E:XmTm}
  \left\|X^{m}u\right\|_{1}\lesssim\sum_{\ell=0}^{m}\left\|T^{m}u \right\|_{1}
\end{align}
 for all $u\in\Lambda^{0,1}(\overline{\Omega})$ and $m\in\mathbb{N}$ with a constant independent of $u$. 
\end{lemma}

\begin{proof}
  Inequality \eqref{E:Tm} for $m=1$ is (ii) of Lemma \ref{L:T}. Suppose that \eqref{E:Tm} holds for any $m<j$ for some $j>1$.  Then
  \begin{align*}
    \left\|\left(X^{j}-T^{j} \right)u \right\|_{1}
    &\leq
    \left\|\left(X^{j-1}-T^{j-1}\right)Xu \right\|_{1}+\left\|T^{j-1}\left(X-T\right)u \right\|_{1}\\
    &\lesssim\sum_{\ell=0}^{j-2}\left\|T^{\ell}(Xu)\right\|_{1}+\left\|T^{j-1}\left(X-T\right)u \right\|_{1},
   \end{align*}
   by the induction hypothesis. 
   Writing $T^{\ell}(Xu)=T^{\ell+1}u+T^{\ell}(X-T)u$, it follows that 
   \begin{align}\label{E:Tmproof}
     \left\|\left(X^{j}-T^{j} \right)u \right\|_{1}\lesssim
     \sum_{\ell=0}^{j-1}\left\|T^{\ell} u\right\|_{1}+\sum_{\ell=0}^{j-1}\left\|T^{\ell}(X-T)u\right\|_{1}
   \end{align}
   holds.
  Lemma \ref{L:AB} applied to the second term in \eqref{E:Tmproof} and the induction hypothesis yield
   \begin{align*}
    \sum_{\ell=0}^{j-1} \left\|T^{\ell}\left(X-T\right)u \right\|_{1}&\leq
    \sum_{\ell=0}^{j-1}\Bigl\{ \left\|(X-T)T^{\ell}u \right\|_{1}+\left\|\left[T^{\ell},X-T \right] u\right\|_{1}\Bigr\}\\
    &\lesssim\sum_{\ell=0}^{j-1}\left\|T^{\ell}u \right\|_{1}.
   \end{align*}
   Hence \eqref{E:Tm} has been proved. Inequality \eqref{E:XmTm} follows straightforwardly from \eqref{E:Tm} after writing
   $X^{m}=(X^{m}-T^{m})+T^{m}$.
\end{proof}

As in the case $k=1$, the form $\varphi=N\dbar f$ is used to pass $X$ by $B$ in two steps. The next result, and extension of Lemma \ref{L:basicphiestimates}, collects estimates on the derivatives of $\varphi$ that arise in the commutation process.

\begin{lemma}\label{L:mphiestimates}
  For $f\in C^{\infty}(\overline{\Omega})$ and $\varphi=N\dbar f$, the inequalities
  \begin{itemize}
    \item[(i)] $Q\left(T^{m}\varphi,T^{m}\varphi\right)\lesssim \|f\|_{m,X}^{2}$,
    \item[(ii)] $\left\|T^{m}\varphi\right\|^{2}+\sum_{j=1}^{n}\left\|\overline{L}_{j}T^{m}\varphi \right\|^{2}
       +\left\|\overline{\cn}T^{m}\varphi \right\|^{2}\lesssim \|f\|_{m,X}^{2}$,
       \item[(iii)] $\left\|T^{m}\varphi\right\|_{1}^{2}\lesssim\left\|T^{m+1}\varphi\right\|^{2}+\|f\|_{m,X}^{2}$
  \end{itemize}
  hold for all  $m\in\{0,\dots,k-1\}$, with a constant independent of $f$.
\end{lemma}

\begin{proof}
  In the case of $m=0$, Lemma \ref{L:mphiestimates} corresponds to Lemma \ref{L:basicphiestimates}. Suppose (i)--(iii) hold true for all $m<j$ for some $j>1$.  Then
  \begin{align*}
     Q\left(T^{j}\varphi, T^{j}\varphi\right)&=
     \left\|\dbar T^{j}\varphi\right\|^{2}+\left\|\dbarstar T^{j}\varphi\right\|^{2}\\
     &\lesssim
      \left\|\dbar X^{j}\varphi\right\|^{2}+\left\|\vartheta X^{j}\varphi\right\|^{2}+\left\|\left(X^{j}-T^{j}\right)\varphi\right\|_{1}^{2}\\
      &\lesssim
       \left\|\dbar X^{j}\varphi\right\|^{2}+\left\|\vartheta X^{j}\varphi\right\|^{2}+\sum_{\ell=0}^{j-1}\left\|T^{\ell}\varphi \right\|_{1}^{2}
  \end{align*}
  by \eqref{E:Tm}. Next, commuting $X$ by $\dbar$ and $\vartheta$ yields
  \begin{align*}
     \left\|\dbar X^{j}\varphi\right\|^{2}+\left\|\vartheta X^{j}\varphi\right\|^{2}
     &\lesssim
     \left\|X^{j}\dbarstar\varphi\right\|^{2}+\left\|\left[\dbar,X^{j} \right]\varphi \right\|^{2}+\left\|\left[\vartheta,X^{j} \right]\varphi \right\|^{2}\\
     &\lesssim
     \left\|X^{j}f\right\|^{2}+\|X^{j}Bf\|^{2}+\sum_{\ell=0}^{j-1}\left\|X^{\ell}\varphi\right\|_{1}^{2},
  \end{align*}
  where in the last step Lemma \ref{L:AB} was used. It now follows from \eqref{E:XmTm}  that
  \begin{align}\label{E:QTphiestimate}
    Q\left(T^{j}\varphi, T^{j}\varphi\right)\lesssim
    \left\|X^{j}f\right\|^{2}+\|X^{j}Bf\|^{2}+\sum_{\ell=0}^{j-1}\left\|T^{\ell}\varphi \right\|_{1}^{2}
  \end{align}
  The induction hypothesis on (ii) and (iii) and \eqref{E:inductionhypo} complete the proof of (i) for $m=j$. The remaining parts, (ii) and (iii), are shown as in Lemma \ref{L:basicphiestimates}.
\end{proof}

\begin{lemma}\label{L:Tkphireduction} 
  For $f\in C^{\infty}(\overline{\Omega})$ and $\varphi=N\dbar f$, the inequality
  \begin{align*}
    \left\|X^{k}Bf \right\|^{2}\lesssim\left\|T^{k}\varphi \right\|^{2}+\|f\|_{k,X}^{2}
  \end{align*}
  holds with a constant independent of $f$.
\end{lemma}

\begin{proof}
  Using arguments analogous to the ones preceding  \eqref{E:XBfXphi}, with $X^{k}Bf$ in place of $XBf$, it follows that
  \begin{align}\label{E:XBfXphik}
    \left\|X^{k}Bf \right\|^{2}&\lesssim \left\|X^{k}f\right\|^{2}+\left|\left(X^{k}\dbarstar\varphi , X^{k}Bf\right)\right|\notag\\
    &\lesssim \left\|X^{k}f\right\|^{2}+\left|\left(\vartheta X^{k}\varphi , X^{k}Bf\right)\right|
    +\left|\left([X^{k},\vartheta]\varphi , X^{k}Bf\right)\right|.
  \end{align}
  For the last term in \eqref{E:XBfXphik}, apply Lemma \ref{L:AB} followed by \eqref{E:XmTm} to obtain
  \begin{align*}
    \left|\left([X^{k},\vartheta]\varphi , X^{k}Bf\right)\right|\lesssim \sum_{\ell=0}^{k-1}\left\|X^{k}\varphi\right\|_{1}\cdot\|X^{k}Bf\|
    \lesssim\sum_{\ell=0}^{k-1}\left\|T^{k}\varphi\right\|_{1}\cdot\|X^{k}Bf\|.
  \end{align*}
  For the second term in \eqref{E:XBfXphik}, proceeding analogously to the arguments subsequent to \eqref{E:XBfXphi} it follows that 
  \begin{align*}
    \left|\left(\vartheta X^{k}\varphi , X^{k}Bf\right)\right|&\lesssim
    \left|\left(\dbarstar T^{k}\varphi , X^{k}Bf\right)\right|+\left\|\left(X^{k}-T^{k}\right)\varphi \right\|_{1}\cdot\left\|X^{k}Bf \right\|\\
    &\lesssim 
    \left\|T^{k}\varphi \right\|\cdot\left\|Bf \right\|_{k}+\sum_{\ell=0}^{k-1}\left\|T^{\ell}\varphi \right\|_{1}\cdot\left\|X^{k}Bf \right\|.
  \end{align*}
  Lemma \ref{L:XkBfreduction} and part (iii) of Lemma \ref{L:mphiestimates} then yield
  \begin{align*}
     \left\|X^{k}Bf \right\|^{2}\lesssim 
     \left\|T^{k}\varphi \right\|^{2}+\left\|f\right\|_{k,X}^{2}.
  \end{align*}
\end{proof}

The following result concludes the proof of \eqref{E:ktoprove}.

\begin{lemma}
  For $f\in C^{\infty}(\overline{\Omega})$ and $\varphi=N\dbar f$, the inequality
  \begin{align*}
    \left\|T^{k}\varphi\right\|^{2}\lesssim\|f\|_{k,X}^{2}
  \end{align*}
  holds with a constant independent of $f$.
\end{lemma}

\begin{proof}
  For $\epsilon>0$, the compactness estimates \eqref{E:compact} gives
  \begin{align*}
    \left\|T^{k}\varphi\right\|^{2}&\leq \epsilon Q\left(T^{k}\varphi,T^{k}\varphi \right)+C(\epsilon)\left|T^{k}\varphi \right|_{-1}^{2}\\
    &\leq 
    \epsilon Q\left(T^{k}\varphi,T^{k}\varphi \right)+C_{1}(\epsilon)\left\|f\right\|_{k-1,X}^{2}
  \end{align*}
  by (ii) of Lemma \ref{L:mphiestimates}. It was shown in the computation leading up to \eqref{E:QTphiestimate} that
  \begin{align*}
    Q\left(T^{k}\varphi,T^{k}\varphi\right)&\lesssim \left\|X^{k}f\right\|^{2}+\left\|X^{k}Bf\right\|^{2}  
   + \sum_{\ell=0}^{k-1}\left\|T^{\ell}\varphi \right\|_{1}^{2}\\
    &\lesssim \left\|f\right\|_{k,X}^{2}+\left\|X^{k}Bf\right\|^{2}  
  \end{align*}
  holds, where in the second estimate (iii) of Lemma \ref{L:mphiestimates} was used. Applying Lemma \ref{L:Tkphireduction} and then choosing
  $\epsilon>0$ sufficiently small yields the claim.
\end{proof}

\medskip

\section{Remarks}\label{S:Remarks}  
\subsection*{The spaces $H_{X}^{k}(\Omega)$} 

Let $\Omega\subset\mathbb{C}^{n}$ be a smoothly bounded domain, $r$ a defining function for $\Omega$. For $X$ as in \eqref{D:badtangential}, define
\begin{align*}
  H_{X}^{\infty}(\Omega)=\bigcap_{k=0}^{\infty}H_{X}^{k}(\Omega).
\end{align*}
If $\Omega$ satisfies the hypotheses of Theorem \ref{T:MainTheorem}, then the latter and the Sobolev embedding theorem imply that
the Bergman projection maps $H_{X}^{\infty}(\Omega)$ to $C^{\infty}(\overline{\Omega})\cap\mathcal{O}(\Omega)$. That this is an improvement over the result by Kohn and Nirenberg \cite[Theorem 2]{KohNir65}) can be seen by examples similar to the one subsequent to Theorem \ref{T:MainTheorem}: for given $g\in C^{\infty}(\overline{\Omega})$ and $\alpha>-\frac{1}{2}$, set
$$ f(z)=\left(-r(z)\right)^{\alpha}g(z).$$
Then, since $\alpha>-\frac{1}{2}$, $f\in L^{2}(\Omega)$. Moreover, the identity
$$X^{k}f(z)=\left(-r(z)\right)^{\alpha}X^{k}g(z)$$
yields $f\in H_{X}^{k}(\Omega)$ for all $k\in\mathbb{N}$. However, for $\alpha\notin\mathbb{N}_{0}$ and $g$ not identically zero on $b\Omega$, it follows that  $f\notin H^{\ell}(\Omega)$ for $\ell\in\mathbb{N}$ with $\ell\geq\alpha+\frac{1}{2}$. Thus $f\in H_{X}^{\infty}(\Omega)\setminus
C^{\infty}(\overline{\Omega})$ while $Bf\in C^{\infty}(\overline{\Omega})$.

\medskip

Theorem \ref{T:MainTheorem} gives a collection of different estimates since the spaces $H_{X}^{k}(\Omega)$ depend on the choice of $X$ (i.e., on the choice of the defining function). E.g., consider the upper half plane
\begin{align}\label{D:upperhalfplane}
  \mathbb{H}&=\{(x,y)\in\mathbb{R}^{2}:r(x,y)=-y<0\}\\
  &=\{(x,y)\in\mathbb{R}^{2}:\rho(x,y)=-(1+x^{2})y<0\}\notag
\end{align}
and the vector fields $X$ and $\widetilde{X}$, defined by \eqref{D:badtangential} for $r$ and $\rho$ respectively, 
\begin{align*}
  X
  =\frac{1}{2i}\frac{\partial}{\partial x},\qquad
  \widetilde{X}
  =
  \frac{1}{2i}\left(-2xy\frac{\partial}{\partial y}+(1+x^{2})\frac{\partial}{\partial x} \right).
\end{align*}
Let $f(x,y)=y\sin(y^{-3})\chi(x,y)$ where $\chi\in C_{c}^{\infty}(\overline{\mathbb{H}})$ is identically $1$ near the origin. Then $f\in L^{2}(\mathbb{H})$ since $|\sin(y^{-3})|\leq 1$ for $y>0$. Furthermore, it follows from the equality
\begin{align*}
  X^{k}f(x,y)=y\sin(y^{-3})X^{k}\chi(x,y)\qquad\sjump\forall\sjump k\in\mathbb{N}
\end{align*}
that $f\in  H_{X}^{k}(\mathbb{H})$ for all $k\in\mathbb{N}$, and hence $f\in H_{X}^{\infty}(\mathbb{H})$. However, an easy calculation shows that
$\widetilde{X}(\sin(y^{-3}))y\chi(x,y)$ fails to be $L^{2}$-integrable near the origin while $\widetilde{X}(y\chi(x,y))\sin(y^{-3})\in L^{2}(\mathbb{H})$. Thus $f\notin H_{\widetilde{X}}^{1}(\mathbb{H})$.

\subsection*{Estimates for other vector fields}
  The essential property of the vector field $X$ for the proof of Theorem \ref{T:MainTheorem} is
  \begin{align}\label{E:spanning}
    T(b\Omega)=\text{span}\{L_{1},\dots,L_{n},\overline{L}_{1},\dots,\overline{L}_{n},X\}
  \end{align}
  The tangentiality of $X$ ensures the construction of $T$ in Lemma \ref{L:T} as well as the validity of the various integration by parts arguments whereas the spanning property in \eqref{E:spanning} guarantees that all derivatives can be converted to $X$ and barred derivatives when measured in $L^{2}$ -- see Proposition \ref{P:basicLestimate} and Corollary \ref{C:basicSobolev1estimate}.
  
  As only property \eqref{E:spanning} is essential, Theorem \ref{T:MainTheorem} holds also for any tangential vector field $Y$, satisfying \eqref{E:spanning}, in place of $X$, e.g.,
  \begin{align*}
    Y=aX+\sum_{j=1}^{n}\left(b_{j}L_{j}+c_{j}\overline{L}_{j}\right),
  \end{align*}
  where $a, b_{j}, c_{j}\in C^{\infty}(\overline{\Omega})$ and $a(z)\neq 0$ for all $z\in b\Omega$. If the function $a$ vanishes on $b\Omega$, then  
  property \eqref{E:spanning} is violated and in fact both Proposition \ref{P:basicLestimate} and Corollary \ref{C:basicSobolev1estimate} fail to be  
  true. The following example shows that the tangentiality of $Y$ is also essential.
 
\subsubsection*{Example} 
Consider the upper half plane $\mathbb{H}$ as in \eqref{D:upperhalfplane}. For a given constant $\alpha\in(-\frac{1}{2},0)$, let
\begin{align*}
  f(x,y)=|x|^{\alpha}\chi(x,y),
\end{align*}
where $\chi\in C_{c}^{\infty}(\overline{\mathbb{H}})$ is non-negative and satisfies $\chi=1$ on $R=[-1,1]\times[0,1]$. Then both $f$ and its normal derivative, $\frac{\partial}{\partial y}f(x,y)=|x|^{\alpha}\chi_{y}(x,y)$, belong to $L^{2}(\mathbb{H})$. However, $\frac{\partial}{\partial z}Bf$ fails to be square-integrable near the origin. In the following, we give a short sketch on how to prove this.

First recall that
$$(Bf)(z)=-\frac{1}{\pi}\int_{\mathbb{H}}\frac{f(w)}{(z-\bar{w})^{2}}\;dV(w),$$
so that
$$\frac{\partial}{\partial z}(Bf)(z)=\frac{2}{\pi}\int_{\mathbb{H}}\frac{f(w)}{(z-\bar{w})^{3}}\;dV(w).$$
Also,  note that it suffices to show that the imaginary part of $\frac{\partial}{\partial z}Bf$ is not in $L^{2}(D_{\epsilon})$ , where 
$D_{\epsilon}=\{(x,y)\in\mathbb{R}^{2}:y\in(0,\epsilon], 0\leq x\leq\epsilon y\}$, for some positive $\epsilon\leq\frac{1}{2}$. A straightforward computation
shows that
\begin{align*}
  \frac{\partial}{\partial z}(Bf)(z)=\frac{2}{\pi}\int_{R}\frac{|\text{Re}(w)|^{\alpha}}{(z-\bar{w})^{3}}\;dV(w)\;\;\;\text{mod}\;L^{2}(D_{\epsilon}).
\end{align*}
Write $w=u+iv$ and integrate with respect to $v$ to obtain
\begin{align*}
  \frac{\partial}{\partial z}(Bf)(z)&=\frac{1}{i\pi}\int_{-1}^{1}|u|^{\alpha}\left((x-u)+iy \right)^{-2}\;du\;\;\;\text{mod}\;L^{2}(D_{\epsilon})\\
  &=-\frac{i}{\pi}\int_{-1}^{1}|u|^{\alpha}\frac{\left((x-u)-iy \right)^{2}}{\left|\left((x-u)+iy \right)^{2} \right|^{2}}\;du\;\;\;\text{mod}\;L^{2}(D_{\epsilon}).
\end{align*}
Therefore,
\begin{align*}
\text{Im}\left(\frac{\partial}{\partial z}(Bf)(z) \right)&=-\frac{1}{\pi}\int_{-1}^{1}|u|^{\alpha}
\frac{(x-u)^{2}-y^{2}}{\left((x-u)^{2}+y^{2} \right)^{2}}\;du\;\;\;\text{mod}\;L^{2}(D_{\epsilon})\\
&=:\text{I}(z)\;\;\;\text{mod}\;L^{2}(D_{\epsilon}),
\end{align*}
and it remains to be shown that $\text{I}\notin L^{2}(D_{\epsilon})$. It follows from a coordinate change in the $x$-variable that
\begin{align*}
  \int_{D_{\epsilon}} \left(\text{I}(z)\right)^{2}\;dV(z)&=\int_{0}^{\epsilon}\int_{0}^{y\epsilon}\left(\text{I}(x,y)\right)^{2}\;dx\;dy\\
  &=\int_{0}^{\epsilon}\int_{0}^{\epsilon}\left(\text{I}(ty,y) \right)^{2}y\;dt\;dy.
\end{align*}
Another change of variable (in the $u$-coordinate) yields
\begin{align*}
  \text{I}(ty,y)&=\frac{1}{\pi}\int_{-1}^{1}|u|^{\alpha}\frac{y^{2}-(ty-u)^{2}}{\left((ty-u)^{2}+y^{2} \right)}\;du\\
  &=\frac{y^{\alpha-1}}{\pi}\int_{-\frac{1}{y}}^{\frac{1}{y}}|\tau|^{\alpha}\frac{1-(t-\tau)^{2}}{\left((t-\tau)^{2}+1 \right)^{2}}\;d\tau
  =:y^{\alpha-1}\cdot\tilde{\text{I}}(t,y).
\end{align*}
However, some further calculation shows that $\tilde{\text{I}}(t,y)$ is strictly positive if $t$ and $y$ are sufficiently close to 0. Hence, for sufficiently small $\epsilon>0$ there exists a constant $C>0$ such that
\begin{align*}
  \int_{D_{\epsilon}} \left(\text{I}(z)\right)^{2}\;dV(z)\geq C\int_{0}^{\epsilon}\int_{0}^{\epsilon}y^{2\alpha-1}\;dt\;dy,
\end{align*}
which does not exist since $\alpha<0$. Therefore, $\frac{\partial}{\partial z} Bf$ is not square-integrable on $\mathbb{H}$ although 
$f\in H_{\frac{\partial}{\partial y}}^{1}(\mathbb{H})$.

\medskip

\section*{Appendix}

On the ball in $\mathbb{C}^{n}$, the Bergman operator commutes with the vector field $X$ defined by \eqref{D:badtangential}. This fact seems likely to be known, especially to mathematicians working in Lie theory, but the authors were unable to find a statement of this result in the literature. 

Let $D=\bigl\{z\in\mathbb{C}^{n}:r(z)=\sum_{j=1}^{n}|z_{j}|^{2}-1<0\bigr\}$ be the unit ball. The Bergman kernel can be explicitly computed on $D$ \cite{Range86}, 
resulting in the expression
\begin{align*}
 Bf(z)=\frac{n!}{\pi^{n}}\int_{D}\frac{f(w)}{(1-\langle z,w\rangle)^{n+1}}\;dV(w),\qquad f\in L^{2}(D),\;z\in D,
\end{align*}
for the Bergman operator $B$ on $D$. To clarify the (elementary) argument below, subscript the vector field $X$, defined by  \eqref{D:badtangential}, with $z$ or $w$ to indicate whether $X$ acts on the free variables $z$ or the integration variables $w$. 
Since
\begin{align*}
  X_{z}\left(\langle z,w\rangle\right)=\sum_{j=1}^{n}z_{j}\frac{\partial}{\partial z_{j}}\left(\langle z,w\rangle\right)
  =\sum_{j=1}^{n}\bar{w}_{j}\frac{\partial}{\partial\bar{w}_{j}}\left(\langle z,w\rangle \right)
  =\overline{X}_{w}\left(\langle z,w\rangle \right),
\end{align*}
it follows that for $f\in C^{\infty}(\overline{D})$
\begin{align*}
  X_{z}Bf(z)
 &=\frac{n!}{\pi^{n}}\int_{\Omega}f(w)\cdot X_{z}\left(\frac{1}{(1-\langle z,w\rangle)^{n+1}} \right)\;dV(w)\\
 &=\frac{n!}{\pi^{n}}\int_{\Omega}f(w)\cdot \overline{X}_{w}\left(\frac{1}{(1-\langle z,w\rangle)^{n+1}} \right)\;dV(w).
\end{align*}
Since $X$ is self-adjoint, integration by parts then yields
\begin{align*}
  X_{z}Bf(z)=\frac{n!}{\pi^{n}}\int_{D}\frac{X_{w}\left(f(w)\right)}{(1-\langle z,w\rangle)^{n+1}}\;dV(w)=B(Xf)(z).
\end{align*}
Thus, $[B,X]=0$ on $C^{\infty}(\overline{D})$.

\vskip .5cm
\bibliographystyle{plain}
\bibliography{HerMcN10}  
\end{document}